\documentclass[12 pt]{article}%
\usepackage{amsmath, amsfonts, amsthm, color,latexsym}
\usepackage{amsmath, ulem}
\usepackage{amsfonts}
\usepackage{amssymb}
\usepackage{color, soul}
\usepackage{xcolor}
\usepackage[all]{xy}
\usepackage{graphicx}%
\providecommand{\U}[1]{\protect\rule{.1in}{.1in}}
\allowdisplaybreaks[4]
\newtheorem{theorem}{Theorem}[section]

\newtheorem{corollary}[theorem]{Corollary}
\newtheorem{example}[theorem]{Example}
\newtheorem{examples}[theorem]{Examples}
\newtheorem{remark}[theorem]{Remark}

\newtheorem{lemma}[theorem]{Lemma}
\newtheorem{final remark}[theorem]{Final Remark}
\newtheorem{definition}[theorem]{Definition}
\allowdisplaybreaks[4]

\newcommand{\sol}[1]{\text{sol}(#1)}

\newcommand {\cvfe} {\overset{\omega^\ast}{\rightarrow}}

\newcommand {\N} {\mathbb{N}}
\newcommand{\norma}[1]{\| #1 \|}
\newcommand{\conj}[2]{\left \{ {#1} \, : \, {#2} \right \}}

\begin{document}

\title{Disjoint $p$-convergent operators and their adjoints}
\author{Geraldo Botelho\thanks{Supported by Fapemig grants PPM-00450-17,  RED-00133-21 and APQ-01853-23.}\, ,  Luis Alberto Garcia and  Vinícius C. C. Miranda\thanks{Supported by CNPq Grant 150894/2022-8 and Fapemig grant APQ-01853-23\newline 2020 Mathematics Subject Classification: 46B42, 47B65.\newline Keywords: Banach lattices, positive operators, disjoint $p$-convergent operators, order weakly compact operators. }}
\date{}
\maketitle

\begin{abstract} First we give conditions on a Banach lattice $E$ so that an operator $T$ from $E$ to any Banach space is disjoint $p$-convergent if and only if $T$ is almost Dunford-Pettis. Then we study when adjoints of positive operators between Banach lattices are disjoint $p$-convergent. For instance, we prove that the following conditions are equivalent for all Banach lattices $E$ and $F$: (i) A positive operator $T \colon E \to F$ is almost weak $p$-convergent if and only if  $T^*$ is disjoint $p$-convergent; (ii) $E^*$ has order continuous norm or $F^*$ has the positive Schur property of order $p$. Very recent results are improved, examples are given and applications of the main results are provided.

\end{abstract}

\section{Introduction and background}

Recall that a linear operator between Banach spaces is said to be completely continuous, or a Dunford-Pettis operator, if it sends weakly null sequences to norm null sequences. Notions alike have been considered in the mathematical literature. For instance, in \cite{sanchez} a linear operator from a Banach lattice to a Banach space is defined to be almost Dunford-Pettis if it sends disjoint weakly
null sequences to norm null sequences; or equivalently, if it sends positive disjoint weakly null sequences to norm null sequences. Almost Dunford-Pettis operators have attracted the attention of many experts, for recent developments see \cite{aqzelbour, aqz, ardakani, galmir}. Summability properties also play a major role in the subject, for example, %the notion of $p$-convergent operators, or a Dunford-Pettis operator of order $p$, was introduced in \cite{castillo}. Letting
for $1 \leq p < \infty$, in \cite{castillo} a linear operator between Banach spaces is defined to be $p$-convergent (or to be a Dunford-Pettis operator of order $p$) if it sends weakly $p$-summable sequences to norm null sequences. The lattice counterpart of this class was introduced in \cite{fourie}: a linear operator from a Banach lattice to a Banach space is disjoint $p$-convergent if it sends disjoint weakly $p$-summable sequences to norm null sequences; or equivalently if it sends positive disjoint weakly $p$-summable sequences to norm null sequences. Such operators were recently studied in \cite{ali, ardakani2, ardakani}. In the latter reference, disjoint $p$-convergent operators were studied under the name of almost $p$-convergent operators. In this note we follow the nomenclature coined by Zeekoei and Fourie \cite{fourie}.

Since weakly $p$-summable sequences are weakly null, every almost Dunford-Pettis operator is disjoint $p$-convergent. The converse is not true: for $1<p\leq\frac{3}{2}$, $p^\ast := p/(p-1)$ and  $\frac{3}{2}<q<3$, we have $3\leq p^{\ast}$, so  $\ell_{q}$ has the Schur property of order $p$ by \cite[Example 3.9]{fourie}, that is, the identity operator ${\rm id}_{\ell_{q}}\colon \ell_{q}\rightarrow\ell_{q}$ is a disjoint $p$-convergent operator. Working with the canonical unit vectors $(e_n)_n$ we get that it is not not almost Dunford-Pettis. In Section \ref{sec-adp} we investigate conditions under which disjoint $p$-convergent operators are almost Dunford-Pettis. In particular, we prove that if the dual $E^*$ of the Banach lattice $E$ has cotype $p$ or $E$ has the disjoint property of order $p$ (meaning that bounded disjoint $E$-valued sequences are weakly $p$-summable), then every disjoint $p$-convergent operator from $E$ to any Banach space is almost Dunford-Pettis. In Section 3, we focus on conditions so that the adjoint of an operator is disjoint $p$-convergent. The first result characterizes when the adjoint of almost limited operators are disjoint $p$-convergent. After proving several characterizations of operators having disjoint $p$-convergent adjoints, we prove a result on adjoints of order weakly compact operators which generalizes a result proved recently in  \cite{ali}. In the final result we show that the following conditions are equivalent for all Banach lattices $E$ and $F$: (i) A positive operator $T \colon E \to F$ is almost weak $p$-convergent if and only if  $T^*$ is disjoint $p$-convergent; (ii) $E^*$ has order continuous norm or $F^*$ has the positive Schur property of order $p$. The relationships between this result and \cite[Theorem 5.11]{ali} are discussed.

We refer the reader to \cite{alip, meyer} for background on Banach lattices and to \cite{fabian} for Banach space theory. Throughout this paper, $X, Y$ denote Banach spaces and $E, F$ denote Banach lattices. We denote by $B_X$ the closed unit ball of $X$ and by $X^*$ its topological dual. For a subset $A \subseteq E$, $\sol{A}$ denotes the solid hull of $A$ and $E^+$ denotes the positive cone of $E$. Unless said otherwise $p$ and $q$ will denote real numbers such that $p,q \geq 1$ and for every $1 \leq p < \infty$, we denote its conjugate by $p^\ast$, that is, $1/p + 1/p^\ast = 1.$ By ${\rm id}_X$ we mean the identity operator on $X$. Linear operators are always assumed to be bounded.

\section{When every disjoint $p$-convergent is almost Dunford-Pettis} \label{sec-adp}

As announced in the Introduction, the purpose of this section
is to give conditions on a Banach lattice $E$ so that  every disjoint $p$-convergent operator defined on $E$ is almost Dunford-Pettis. In order to do so we recall some terminology and prove some preliminaries results.
Given an operator $T \colon E \to X$, according to  \cite[p.\,192]{meyer} consider the lattice seminorm $q_T$ on $E$ defined by
$$ q_T(x) = \sup \conj{\norma{T(y)}}{|y| \leq |x|}, \,x \in E.$$
We have $\norma{T(x)} \leq q_T(x) \leq \norma{T} \cdot \norma{x}$ for every $x \in E$ (see \cite[Section 2]{aqzelbour}).
We begin by characterizing disjoint $p$-convergent  operators on $E$ by means of $q_T$.

\begin{theorem} \label{teo1} The following are equivalent for a linear operator $T\colon E \to X$: \\
    {\rm (a)} $T$ is disjoint $p$-convergent. \\
    {\rm (b)} $\norma{T(x_n)} \longrightarrow 0$ for every positive disjoint weakly $p$-summable sequence $(x_n)_n$ in $E$. \\
    {\rm (c)} $q_T(x_n) \longrightarrow 0$ for every positive disjoint weakly $p$-summable sequence $(x_n)_n$ in $E$. \\
    {\rm (d)} $q_T(x_n) \longrightarrow 0$ for every disjoint weakly $p$-summable sequence $(x_n)_n$ in $E$.
\end{theorem}

\begin{proof}
    The equivalence (a)$\Leftrightarrow$(b) follows from \cite[Proposition 4.6]{fourie}.

    (a)$\Rightarrow$(d) Let $(x_n)_n$ be a disjoint weakly $p$-summable sequence in $E$. By the definition of $q_T$, there exists a sequence $(y_n)_n \subseteq E$ such that $|y_n| \leq |x_n|$ and $q_T(x_n) \leq 2 \norma{T(y_n)}$ for every $n \in \N$. As $(x_n)_n$ is a disjoint weakly $p$-summable sequence, the sequence $(|x_n|)_n$ is weakly $p$-summable as well by \cite[Proposition 2.2]{fourie}. Since the sequences $(y_n^+)_n$ and $(y_n^-)_n$ are also disjoint weakly $p$-summable, we have
    $$ q_T(x_n) \leq 2 \norma{T(y_n)} \leq 2 (\norma{T(y_n^+)} + \norma{T(y_n^-)}) \longrightarrow 0. $$

   The implication  (d)$\Rightarrow$(c) is immediate. And (c)$\Rightarrow$(b) is easy because if $(x_n)_n$ is a positive disjoint weakly $p$-summable sequence in $E$, then
    $ \norma{T(x_n)} \leq q_T(x_n) \longrightarrow 0. $
\end{proof}

Let us explain our strategy to prove that every disjoint $p$-convergent is almost Dunford-Pettis. Recall that a linear operator $T\colon E \to X$ is said to be order weakly compact whenever $T([0,x])$ is a relatively  weakly  compact subset of $X$ for every $x \in E^+$, or, equivalently, if $\norma{T(x_n)} \longrightarrow 0$ for every order bounded disjoint sequence $(x_n)_n \subset E$ (see \cite[Theorem 5.57]{alip}).

%for each $f = (f_j) \in \ell_1 = c_0'$, $$ \sum_{j=1}^\infty |f(e_n)|^p = \sum_{j=1}^\infty |f_j|^p < \infty $$
%Thus, not every order weakly compact operator is disjoint $p$-convergent. Nevertheless, as we claimed before, every disjoint $p$-convergent operator is order weakly compact:

\begin{lemma} \label{lema1} Disjoint $p$-convergent operators are order weakly compact.
    \end{lemma}
\begin{proof} Let $T\colon E \to X$ be a disjoint $p$-convergent operator and let
     $(x_n)_n$ be an order bounded disjoint sequence in $E$, say $|x_n| \leq x \in E$ for every $n \in \N$. For each $x^\ast \in E^\ast$,
    $$ \sum_{k=1}^n |x^\ast(x_n)| \leq \sum_{k=1}^n |x^\ast|(|x_n|) = |x^\ast| \left ( \displaystyle\bigvee_{k=1}^n |x_n| \right ) \leq |x^\ast|(x), $$
     so $(x^\ast(x_n)) \in \ell_1$ for every $x^\ast \in E^\ast$, proving that  $(x_n)_n$ is weakly $p$-summable. Since $T$ is a disjoint $p$-convergent operator, we have $\norma{T(x_n)} \longrightarrow 0$, thus $T$ is order weakly compact.
\end{proof}

The converse of the lemma above is not true:

\begin{example}\rm Since $c_0$ has order continuous norm, ${\rm id}_{c_0}$ is order weakly compact \cite[p.\,319]{alip}. However, ${\rm id}_{c_0}$ is not disjoint $p$-convergent for any $1 \leq p < \infty$: the canonical unit vectors  $(e_n)_n$ form a disjoint weakly $p$-summable sequence in $c_0$ but $\norma{{\rm id}_{c_0}(e_n)} = \norma{e_n} = 1$ for every $n \in \N$.
\end{example}

It was proved in \cite[Theorem 2.2]{aqzelbour} that a linear operator $T\colon E \to X$ is almost Dunford-Pettis if and only if $T$ is order weakly compact and each relatively weakly compact  subset $W$ of $E$ is $q_T$-almost order bounded, meaning that for every $\varepsilon > 0$ there exists $u \in E^+$ such that $q_T((|y| - u)^+) \leq \varepsilon$ for every $y \in W$. Therefore, according to Lemma \ref{lema1}, in order to prove that every disjoint $p$-convergent operator  $T$ on a Banach lattice $E$ is almost Dunford-Pettis, it is enough to check that every relatively weakly compact  subset of $E$ is $q_T$-almost order bounded. This will be our strategy. We shall give two conditions on $E$ under any of them the desired implication holds. To describe the first condition, recall that a Banach space $X$ has type $q$, $1 \leq q \leq 2$ (cotype $p$, $2 \leq p < \infty$, respectively), if there is a constant $C > 0$ such that for every finite subset $\{x_1, \dots, x_n \} \subseteq X$,
$$ \left ( \int_0^1 \left\|\sum_{k=1}^n r_k(t) x_k\right\|^2  \right)^{1/2} \leq C\cdot \left ( \sum_{k=1}^n \norma{x_k}^q \right )^{1/q},$$
$$\left ( \sum_{k=1}^n \norma{x_k}^p \right )^{1/p} \leq C \left ( \int_0^1 \left\|\sum_{k=1}^n r_k(t) x_k\right\|^2 \right )^{1/2}, \mbox{ respectively},$$
where $(r_k)_k$ denotes the Rademacher sequence (see \cite[p.\,217]{diestel} or \cite[Definition 6.2.10]{albiac}).

A Banach space $X$ is said to have nontrivial type whenever it has some type $1 < q \leq 2$.
In \cite[Lemma 3.4]{zeekoei}, Fourie and Zeekoei proved that in a Banach lattice $E$ with type $1< q \leq 2$, every disjoint sequence in the solid hull of a relatively weakly compact subset of $E$ is weakly $p$-summable for every $p \geq q^\ast$. In their proof they use the type $q$ of $E$ only to conclude that $E^*$ has cotype $q^*$, so this result holds for a larger class of Banach spaces, exactly with the same proof: %\textcolor{red}{This result, however, holds for a more general class of Banach lattices. Recall that a Banach space $X$ is said to have cotype $p$, $2 \leq p < \infty$, if there exists a constant $D > 0$ such that for every finite subset  $\{x_1, \dots, x_n \} \subset X$,
%$$ \left ( \sum_{k=1}^n \norma{x_k}^p \right )^{1/p} \leq D \left ( \int_0^1 \norma{\sum_{k=1}^n r_k(t) x_k}^2 \right )^{1/2}. $$
%Since, in the proof of \cite[Lemma 3.4]{zeekoei}, the authors assume the Banach lattice $E$ having type $1 < q \leq 2$, in order apply \cite[Proposition 11.10]{diestel} to obtain that $E^\ast$ has cotype $2 \leq q^\ast < \infty$, and then use \cite[Corollary 11.17]{diestel} which only assumes cotype. From this observation, by following the same proof given in \cite[Lemma 3.4]{zeekoei}, we have the following:}

\begin{lemma} \label{lemafourie}
    If $E^\ast$ has cotype $2 \leq p < \infty$, then every disjoint sequence contained in the solid hull of a relatively weakly compact subset of $E$ is weakly $p$-summable in $E$.
\end{lemma}

Since $c_0$ fails to have nontrivial type, one cannot apply \cite[Lemma 3.4]{zeekoei} to $c_0$. Nevertheless, since $c_0^{\ast} = \ell_1$ has cotype $2$,  Lemma \ref{lemafourie} applies for $c_0$ with $p = 2$.

Next we define the second condition we shall impose on $E$.

\begin{definition} \label{djp1}\rm
    A Banach lattice $E$ is said to have the {\it disjoint property of order $p$} if each bounded disjoint sequence in $E$ is weakly $p$-summable.
\end{definition}

%Let us see some remarks concerning this new property:

\begin{examples} \label{djp2}\rm (1) On the one hand, if $E$ has the disjoint property of order $p$, then $E$ has the disjoint property of order $q$ for any $q > p$. On the other hand, $\ell_3$ has the disjoint property of order $3$, but, using the canonical vectors $(e_n)_n$, we see that it fails the disjoint property of order $3/2$. %, because the canonical basis $(e_n)_n$ is a disjoint sequence in $\ell_3$ which is not weakly $\frac{3}{2}$-summable. %%{\normalfont To see that, take $\varphi = (a_j)_j \in \ell_{3/2} = \ell_3'$ which does not belong to $\ell_3$ and note that $\varphi(e_n) = a_n$ for every $n \in \N$.}

\noindent (2) By Lemma \ref{lemafourie}, every reflexive Banach lattice $E$ such that $E^\ast$ has cotype $2 \leq p < \infty$ has the disjoint property of order $p$. For instance, from \cite[Theorem 6.2.14]{albiac} we have the following for any measure $\mu$:

    {\rm (i)} For $1 < p < 2$, $(L_p(\mu))^\ast = L_{p^\ast}(\mu)$ has cotype $p^\ast >2$, so $L_p(\mu)$ has the disjoint property of order $p^{\ast}$.

    {\rm (ii)} For $2 \leq p < \infty$, $(L_p(\mu))^\ast = L_{p^\ast}(\mu)$ has cotype $2$, so $L_p(\mu)$ has the disjoint property of order $2$.

    %\vspace*{0.5em}
%\medskip
\noindent (3) If $E$ has the disjoint property of order $p$, then $E^\ast$ has order continuous norm.  Indeed, if $(x_n)_n$ is a bounded disjoint sequence  in $E$, then $(x_n)_n$ is weakly $p$-summable, hence it is weakly null. It follows that $E^\ast$ has order continuous norm by \cite[Theorem 2.4.14]{meyer}. In particular, infinite dimensional AL-spaces fail the disjoint property of order $p$ for any $p$.

%    \medskip
\noindent(4) If $E$ has an order unit, then $E$ has the disjoint property of order $p$ for every $p$. To see that, let $(x_n)_n \subseteq E$ be a bounded disjoint sequence. Since $E$ has an order unit, there exists $x \in E$ such that $|x_n| \leq x$ for every $n \in \N$. Given $x^\ast \in E^\ast$, we have
        $$ \sum_{k=1}^n |x^\ast(x_n)| \leq \sum_{k=1}^n |x^\ast|(|x_n|) = |x^\ast| \left ( \bigvee_{k=1}^n |x_n| \right ) \leq |x^\ast|(x) $$
         for every $k \in \N$, which gives $(x^\ast(x_n))_n \in \ell_1 \subseteq \ell_p$, proving that %By considering the inclusion $\ell_1 \to \ell_p$, we get that
         $(x_n)_n$ is weakly $p$-summable.

%    \medskip
\noindent (5)  The Banach lattice $c_0$ has the disjoint property of order $1$, hence of order $p$ for every $p$. Indeed, let $(a_n)_n \subseteq c_0$ be a bounded disjoint sequence, say $\norma{a_n}_\infty \leq M$ for every $n \in \N$. We write $a_n = (a_{n,j})_j \in c_0$ for each $n \in \N$. If $|a_{n_0,j_0}| \neq 0$, then $a_{n,j_0} = 0$ for every $n \neq n_0$, which means that, for every $j \in \N$, the set
    $ \{ n \in \N : a_{n,j} \neq 0 \} $
    is a singleton or empty. So, letting
    $ J = \{ j \in \N: \exists n_j \in \N \mbox{ such that } a_{n_j, j} \neq 0 \}, $ we have $a_{n, j} = 0$ for all $n \neq n_j$ and $j \in J$. Therefore, for all
    $f = (f_j)_j \in \ell_1$ and $k \in \mathbb{N}$, we get
        $$ \sum_{n=1}^k |f(a_n)| \leq \sum_{n=1}^k \sum_{j=1}^\infty |f_j| \cdot|a_{n,j}| = \sum_{j=1}^\infty \sum_{n=1}^k |f_j|\cdot |a_{n,j}| = \sum_{j\in J}^\infty |f_j|\cdot |a_{n_j,j}| \leq
        M \cdot \sum_{j=1}^\infty |f_j| < \infty, $$
which  proves that $c_0$ has the disjoint property of order $1$.%, and by the above item $c_0$ has the disjoint property of order $p$ for all $1 \leq p < \infty$.
\end{examples}

%Our next result states the same as Theorem \ref{teo2} under different hypothesis in the Banach lattice $E$.

\begin{theorem}\label{teo3} Let $1 \leq p < \infty$ and let $E$ be a Banach lattice such that either $E^*$ has cotype $p$ or $E$ has the disjoint property of order $p$. Then, a linear operator $T\colon E \to X$ is disjoint $p$-convergent if and only if $T$ is almost Dunford-Pettis.
\end{theorem}
\begin{proof}
    Let $T\colon E \to X$ be a disjoint $p$-convergent operator. It follows from Lemma \ref{lema1} that $T$ is order weakly compact. By \cite[Theorem 2.2]{aqzelbour} is is enough to prove that every relatively weakly compact subset of $E$ is $q_T$-almost order bounded. To accomplish this task, let $W \subseteq E$ be a relatively weakly compact subset of $E$, call $A = \sol{W}$ and let $(y_n)_n$ be a given disjoint sequence in $A$. We claim that $q_T(y_n) \longrightarrow 0$. By the definition of $q_T$ there exists a sequence $(z_n)_n$ such that $|z_n| \leq |y_n|$ and $q_T(y_n) \leq 2 \norma{T(z_n)}$ for every $n \in \N$.

     If $E^*$ has cotype $p$, % first that $E^*$ has cotype $p \geq 2$.  %As we already known, it suffices us to check that every relatively weakly compact subset of $E$ is $q_T$-almost order bounded.
   since $(z_n)_n$ is a disjoint sequence contained in the solid hull of a relatively weakly compact set, we obtain from Lemma \ref{lemafourie} that $(z_n)_n$ is weakly $p$-summable. And if $E$ has the disjoint property of order $p$, since $(z_n)_n$ is a disjoint bounded sequence, we obtain that $(z_n)_n$ is weakly $p$-summable by assumption. So, $(z_n)_n$ is weakly $p$-summable either way.

     As $T$ is a disjoint $p$-convergent operator, we have $ q_T(y_n) \leq 2 \norma{T(z_n)} \longrightarrow 0, $ which gives $q_T(y_n) \longrightarrow 0$.
     Now, by  \cite[Theorem 4.36]{alip} there exists $u \in E^+$ such that
     $$q_T((|x| - u)^+) = q_T({\rm id}_E(|x| - u)^+) < \varepsilon$$
     for every $x \in A$.  In particular, $q_T((|x| - u)^+) < \varepsilon$ for every $x \in W$, which proves that $W$ is $q_T$-almost order bounded. Thus $T$ is an almost Dunford-Pettis operator.
    %uppose now that $E$ has the disjoint property of order $p$.  %By \cite[Theorem 2.2]{aqzelbour}, it suffices us to check that every relatively weakly compact subset of $E$ is $q_T$-almost order bounded. So,
 %Given a relatively weakly compact subset $W$ of $E$, let $A = \sol{W}$. Given a disjoint sequence $(y_n)_n \subseteq A$, we claim that $q_T(y_n) \longrightarrow 0$. Indeed, by the definition of $q_T$, there exists a sequence $(z_n)_n$ such that $|z_n| \leq |y_n|$ and $q_T(y_n) \leq 2 \norma{T(z_n)}$ for all $n \in \N$. As $(z_n)_n$ is a  bounded disjoint sequence in $E$, the hypothesis yields that $(z_n)_n$ is weakly $p$-summable. The remain of the proof follows by the same argument used in the proof of Theorem \ref{teo2}, and therefore $T$ is an almost Dunford-Pettis operator.
\end{proof}

\begin{remark}\rm Let us discuss briefly the assumptions on $E$ in  the theorem above. On the one hand, we can apply the theorem for $E = c_0$ only for $p \geq 2$ ($c_0^* = \ell_1$ has  cotype 2); but using that $c_0$ has the disjoint property of order 1 (Example \ref{djp2}(5)), we can apply the theorem for every $p \geq 1$. On the other hand, we still do not know whether there exists a Banach lattice $E$ such that $E^*$ has cotype $2 \leq p < \infty$ and $E$ fails the disjoint property of order $p$.
\end{remark}
    %\textcolor{red}{We observed before Definition \ref{djp1} that  Theorem \ref{teo2} could be applied for $E = c_0$ and $ p = 2$. However, since $c_0$ has the disjoint property of order $p$ for every $1 \leq p < \infty$ (Example \ref{djp2}-5), Theorem \ref{teo3} gives us that, for every $1 \leq p < \infty$ and every Banach space $X$, a bounded linear operator $T \colon c_0 \to X$ is disjoint $p$-convergent if and only if $T$ is almost Dunford-Pettis. This example shows that we may apply Theorem \ref{teo3} in cases that Theorem \ref{teo2} does not hold.
%    %%%It follows from Remark \ref{djp2}-3 that $C([0,1])$ has the disjoint property of order $p$ but it fails to have nontrivial type (see \cite[Exampĺe 6.2.17]{albiac}). Therefore, $C([0,1])$ is an example of a Banach lattice where we can apply Theorem \ref{teo3} but we cannot apply Theorem \ref{teo2}.
%    We still do not know whether there exists a Banach lattice whose dual has cotype $2 \leq p < \infty$ that fails to have the disjoint property of order $p$. }

As a first application of Theorem \ref{teo3} we show that, under the assumptions of the theorem, the disjointness of positive sequences can be dropped in Theorem \ref{teo1}.

\begin{corollary} \label{cor1}
   Let $E$ be a Banach lattice satisfying one of the assumptions of Theorem {\rm \ref{teo3}}. The following are equivalent for a linear operator $T\colon E \to X$: \\
    {\rm (a)} $T$ is disjoint $p$-convergent. \\
    {\rm (b)} $\norma{T(x_n)} \longrightarrow 0$ for every positive weakly $p$-summable sequence $(x_n)_n$ in $E$. \\
    {\rm (c)} $q_T(x_n) \longrightarrow 0$ for every positive weakly $p$-summable sequence $(x_n)_n$ in $E$.
\end{corollary}

\begin{proof}
    (a)$\Rightarrow$(b) Since $T$ is disjoint $p$-convergent, $T$ is almost Dunford-Pettis by Theorem \ref{teo3}. If $(x_n)_n$ is a positive weakly $p$-summable sequence, then $(x_n)_n$ is weakly null. Therefore $\norma{T(x_n)} \longrightarrow 0$ by the implication 1)$\Rightarrow$3) of \cite[Theorem 2.2]{aqzelbour}.

The implication (b)$\Rightarrow$(c) follows from the same argument used in the proof of Theorem \ref{teo1}, and (c)$\Rightarrow$(a) follows directly from the implication (c)$\Rightarrow$(a) of Theorem \ref{teo1}.
\end{proof}

\section{When the adjoint operator is disjoint $p$-convergent}

Adjoint operators enjoy properties that are not shared by arbitrary operators between dual spaces. For instance, the adjoint of an operator is always weak$^*$-weak$^*$ continuous, a property that does not hold for arbitrary operators between dual spaces. Following this trend, in this section we prove several properties of disjoint $p$-convergent adjoint operators. We focus on investigating the situation where the adjoint of every operator is disjoint $p$-convergent.

Before proceeding to the main results of the section, we give a second application of Theorem \ref{teo3} concerning the scope of this section. For the class of almost limited operators see, e.g., \cite{elbour}.

\begin{corollary}\label{corolario} Let $F$ be a Banach lattice such that $F^*$ has the disjoint property of order $p$. The following are equivalent for a linear operator $T \colon X \longrightarrow F$:\\
%    Let $E$ and $F$ be two Banach lattices such that $F^{\ast}$ has the disjoint property of order $p$ and let $T\colon E \rightarrow F$ be a positive operator. :\\
    {\rm (a)} $T^{\ast} \colon F^{\ast}\rightarrow X^{\ast}$  is disjoint $p$-convergent.\\
    {\rm (b)} $T^{\ast}\colon F^{\ast} \rightarrow X^{\ast}$ is almost Dunford-Pettis.\\
    {\rm (c)} $T$ is almost limited, that is $T^\ast$ maps disjoint weak* null sequences in $F^*$ to norm null sequences in $X^*$.
\end{corollary}
\begin{proof}
    (a)$\Leftrightarrow$(b) follows from Theorem \ref{teo3} and (c)$\Rightarrow$(b) is immediate. Let us prove (a)$\Rightarrow$(c). Given a disjoint weak$^*$ null sequence $(y_n^\ast)_n$ in $F^*$, $(y_n^\ast)_n$ is a bounded disjoint sequence in a Banach lattice with the disjoint property of order $p$, so it is weakly $p$-summable. It follows that $\norma{T^\ast(y_n^\ast)} \longrightarrow 0$ because $T^\ast$ is disjoint $p$-convergent. % and if $(y_n^\ast)_n$ is a disjoint weak* null sequence in $F^\ast$, by the assumption, we obtain that $(y_n^\ast)_n$ is weakly $p$-summable, and so $\norma{T^\ast(y_n^\ast)} \longrightarrow 0$.
\end{proof}

In the next two results we prove characterizations of disjoint $p$-convergent adjoint operators that shall be useful a couple of times. %From now on we shall work with operators between Banach lattices.
%%\textcolor{red}{Estou enunciando o lema abaixo para operadores a valores em espa\c cos de Banach. Por favor confiram se está tudo OK.}

%In order to state our results, we begin by characterizing when an adjoint operator is disjoint $p$-convergent.

\begin{lemma}\label{equivalentes}
The following are equivalent for a linear operator $T \colon X \rightarrow F$:\\
{\rm (a)} $T^{\ast}\colon F^{\ast}\rightarrow X^{\ast}$ is disjoint $p$-convergent.\\
{\rm (b)} $q_{T^{\ast}}(y_{n}^{\ast})\longrightarrow 0$ for every disjoint weakly $p$-summable sequence $(y_{n}^{\ast})_n$ in $F^{\ast}$.\\
{\rm (c)} $q_{T^{\ast}}(y_{n}^{\ast})\longrightarrow 0$ for every positive disjoint weakly $p$-summable sequence $(y_{n}^{\ast})_n$ in $F^{\ast}$.\\
{\rm (d)} $y_{n}^{\ast}(T(x_{n}))\longrightarrow 0$ for every disjoint weakly $p$-summable sequence $(y_{n}^{\ast})_n$ in $F^{\ast}$  and every bounded sequence  $(x_{n})_n$ in $X$.\\
{\rm (e)} $y_{n}^{\ast}(T(x_{n}))\longrightarrow 0$ for every positive  disjoint weakly $p$-summable sequence $(y_{n}^{\ast})_n$ in $F^{\ast}$  and every bounded sequence  $(x_{n})_n$ in $X$.\\
{\rm (f)} Regardless of the Banach lattice $G$ and the linear operator $S\colon G\rightarrow X$, $S^{\ast}T^{\ast}\colon F^{\ast}\rightarrow G^{\ast}$ is a disjoint $p$-convergent operator.\\ % for every linear operator  and every Banach lattice $G$.\\
{\rm (g)} Regardless of the linear operator  $S \colon \ell_{1}\rightarrow X$,  $S^{\ast}T^{\ast}\colon F^{\ast}\rightarrow \ell_{1}^{\ast}$ is a disjoint $p$-convergent operator. % for every linear operator.
\end{lemma}
\begin{proof} First note that the implications (b)$\Rightarrow$(c), (d)$\Rightarrow$(e) and (f)$\Rightarrow$(g) are  immediate.

(a)$\Rightarrow$(b) Let $(y_n^\ast)_n$ be a disjoint weakly $p$-summable sequence in $F^{\ast}$. By the definition of $q_{T^{\ast}}$ there exists a sequence $(z_{n}^{\ast})_n$ in $F^{\ast}$ such that $|z_{n}^{\ast}|\leq |y_{n}^{\ast}|$ and $0\leq q_{T^{\ast}}(y_{n}^{\ast})\leq \|T^{\ast}(z_{n}^{\ast})\|+\frac{1}{n}$ for every $n\in\mathbb{N}$. Since $(y_n^\ast)_n$ is disjoint and weakly $p$-summable, $(|y_n^\ast|)_n$ is weakly $p$-summable by \cite[Proposition 2.2]{fourie}, thus $(z_n^\ast)_n$ is disjoint and weakly $p$-summable as well. Hence
$\norma{T^\ast(z_n^\ast)} \longrightarrow 0$, from which we get  $q_{T^{\ast}}(y_{n}^{\ast})\longrightarrow 0$.

%(b) $\Rightarrow$ (c): Immediate.

(c)$\Rightarrow$(d) Let $(y_{n}^{\ast})_n$ be a disjoint weakly $p$-summable sequence in $F^\ast$ and $(x_{n})_n$ be a bounded sequence in $X$. By \cite[Corollary 2.3]{fourie}, the sequences
$((y_{n}^{\ast})^{+})_n$ and $((y_{n}^{\ast})^{-})_n$  are  positive disjoint  weakly $p$-summable and, by assumption, we have  $q_{T^{\ast}}((y_{n}^{\ast})^{+})\longrightarrow 0$ and $q_{T^{\ast}}((y_{n}^{\ast})^{-})\longrightarrow 0$. Letting $M = \sup\limits_{n \in \N} \norma{x_n}$ we get
     \begin{align*}
         |y_{n}^{\ast}(T(x_{n}))|&=|T^{\ast}(y_{n}^{\ast})(x_{n})|\leq \|T^{\ast}(y_{n}^{\ast})\|\cdot\|x_{n}\|\leq M \|T^{\ast}(y_{n}^{\ast})\|\\
         &\leq M[\|T^{\ast}((y_{n}^{\ast})^{+})\|+\|T^{\ast}((y_{n}^{\ast})^{-})\|]\leq M [q_{T^{\ast}}((y_{n}^{\ast})^{+})+q_{T^{\ast}}((y_{n}^{\ast})^{-})]\longrightarrow 0,
     \end{align*}
     which implies that $y_{n}^{\ast}(T(x_{n}))\longrightarrow 0.$

%(d) $\Rightarrow$ (e): Immediate.

(e)$\Rightarrow$(a) Suppose that $T^\ast$ is not disjoint $p$-convergent. In this case there exists, by \cite[Proposition 4.6]{fourie}, a positive disjoint weakly $p$-summable sequence $(y_n^\ast)_n$ in $F^\ast$ such that $(T^\ast(y_n^\ast))_n$ is not norm null. By passing to a subsequence if necessary, we may assume that there exists $\alpha > 0$ such that $\norma{T^\ast(y_n^\ast)} \geq \alpha$ for every $n \in \N$. For each $n \in \N$, letting $x_n^\ast = \dfrac{T^\ast(y_n^\ast)}{\norma{T^\ast(y_n^\ast)}}$, there exists $x_n \in B_X$ such that $|x_n^\ast(x_n)| \geq 1/2$. So, for every $n \in \N$,
$$ \dfrac{1}{2} \leq |x_n^\ast (x_n)| = \dfrac{|T^\ast(y_n^\ast)(x_n)|}{\norma{T^\ast(y_n^\ast)}} \leq \dfrac{1}{\alpha} |y_n^\ast(T(x_n))|, $$
which contradicts the assumption $y_n^\ast(T(x_n)) \longrightarrow 0$.

(a)$\Rightarrow$(f) Let $S \colon G \to F$  be a linear operator. If $(y_n^\ast)_n$ is a disjoint weakly $p$-summable sequence in $F^*$, then $\norma{T^\ast y_n^\ast} \longrightarrow 0$ because $T^\ast$ is disjoint $p$-convergent, therefore $ \norma{S^\ast T^\ast (y_n^\ast)} \leq \norma{S^\ast} \cdot \norma{T^\ast (y_n^\ast)} \longrightarrow 0, $
proving that $S^\ast T^\ast$ is disjoint $p$-convergent.

%(f) $\Rightarrow$ (g): Immediate.

(g)$\Rightarrow$(d) Let $(y_{n}^{\ast})_n$ be a disjoint weakly $p$-summable sequence in $F^{\ast}$  and let $(x_{n})_n$ be a bounded sequence in $X$. Defining $S \colon \ell_{1}\rightarrow X$ by $S (a)=\sum\limits_{n=1}^{\infty} a_{n}x_{n}$ for every $a=(a_{n})_n \in \ell_{1}$, we obtain that $S$ is a (bounded) linear operator. Thus $S^\ast T^\ast$ is disjoint $p$-convergent by assumption. It follows that
\begin{align*}
    |y_n^\ast (T(x_n))| & = |y_n^\ast (T(S(e_n)))|  \leq \sup \conj{|y_n^\ast (TS(a))|}{a \in B_{\ell_1}} \\
        & = \sup \conj{|(S^\ast T^\ast)(y_n^\ast) (a)|}{a \in B_{\ell_1}}  = \norma{S^\ast T^\ast (y_n^\ast)} \longrightarrow 0.
\end{align*}
\end{proof}

The following lemma, which was inspired by ideas that go back to Dodds and Fremlin \cite{dofre} and de Pagter and Schep \cite{pagter}, is needed to prove our next result.

\begin{lemma} \label{lemaantigo}
   Let $E$ be a Banach lattice and let $(x_n^\ast)_n $ be a positive seminormalized weak* null sequence in $E^\ast$, say $\norma{x_n^\ast} \geq M > 0$ for every $n \in \N$. Then, there exist an increasing sequence $(n_m)_m \subset \N$ and a positive disjoint bounded sequence $(x_n)_n \subset E$ such that $x_{n_m}^\ast (x_m) \geq M(4^{-1} - 2^{-m})$ for every $m \in \N$.
\end{lemma}

\begin{proof}
%Por hipótese $\|x_{n}^{\ast}\|\geq M>0$, para todo $n\in\mathbb{N}$. Ponha
Setting $z_{n}^{\ast}=\frac{x_{n}^{\ast}}{\|x_{n}^{\ast}\|}$ for every $n$, from  %$z_{n}^{\ast}\geq 0$ e  por \cite[p. 182]{positiveoperators}
$$1=\|z_{n}^{\ast}\|=\sup\{z_{n}^{\ast}(z): z\in E^{+} \text{ and } \|z\|=1\}$$
\cite[p.\,182]{alip}, there exists a positive sequence  $(z_{n})_n$ in $E$ such that $\|z_{n}\|=1$ and $z_{n}^{\ast}(z_{n})\geq \frac{1}{2}$  for every $n\in\mathbb{N}$. %Como $x_{n}^{\ast}\xrightarrow{\,\,\omega^{\ast}}0$ então
It is clear that $z_{n}^{\ast}\xrightarrow{\,\,\omega^{\ast}}0$. Putting $n_{1}=1$ we have $z_{n}^{\ast}(z_{n_{1}})\longrightarrow 0$, so there is $n_{2}>n_{1}$ such that $z_{n_{2}}^{\ast}(z_{n_{1}})<\frac{1}{2^{2\times 2+2}}$. The convergence $z_{n}^{\ast}(z_{n_{1}}+z_{n_{2}})\longrightarrow 0$ gives $n_{3}>n_{2}$ such that $z_{n_{3}}^{\ast}(z_{n_{1}}+z_{n_{2}})<\frac{1}{2^{2\times 3+2}}$. Inductively we construct a subsequence $(n_{m})_m$ of $\mathbb{N}$ such that $z_{n_{m}}^{\ast}\left(\sum\limits_{k=1}^{m-1}z_{n_{k}}\right)\leq \frac{1}{2^{2m+2}}$ for every $m\geq 2$. Now define
$$x_{m}=\left(z_{n_{m}}-4^{m}\cdot\displaystyle\sum_{k=1}^{m-1}z_{n_{k}}-2^{-m}\cdot\sum_{k=1}^{\infty}2^{-k}z_{n_{k}}\right)^{+} \text{ for every } m\geq 2.$$
Taking $x_{1}=0$, the sequence $(x_{m})_m$ is positive disjoint in $E$ by \cite[Lemma 2.6]{retbi} or, alternatively, by the proof of \cite[Lemma 2.4]{pagter}. Note that  $0\leq x_{m}\leq z_{n_{m}}$, hence $\|x_{m}\|\leq \|z_{n_{m}}\|=1$ for every $m\in\mathbb{N}$. For $m\geq 2$,
 \begin{align*}
z_{n_{m}}^{\ast}(x_{m})&=z_{n_{m}}^{\ast}\left(z_{n_{m}}-4^{m}\displaystyle\sum_{k=1}^{m-1}z_{n_{k}}-2^{-m}\sum_{k=1}^{\infty}2^{-k}z_{n_{k}}\right)^{+}\\
&\geq z_{n_{m}}^{\ast}\left(z_{n_{m}}-4^{m}\displaystyle\sum_{k=1}^{m-1}z_{n_{k}}-2^{-m}\sum_{k=1}^{\infty}2^{-k}z_{n_{k}}\right)\\
&=z_{n_{m}}^{\ast}(z_{n_{m}})-4^{m}z_{n_{m}}^{\ast}\left(\displaystyle\sum_{k=1}^{m-1}z_{n_{k}}\Big)-2^{-m}z_{n_{m}}^{\ast}\Big(\sum_{k=1}^{\infty}2^{-k}z_{n_{k}}\right)\\&\geq \frac{1}{2}-\frac{4^{m}}{2^{2m+2}}-\frac{1}{2^{m}}=\frac{1}{2}-\frac{1}{4}-\frac{1}{2^{m}}=\frac{1}{4}-\frac{1}{2^{m}},
\end{align*}
from which it follows that %$z_{n_{m}}^{\ast}(x_{m})\geq \frac{1}{4}-\frac{1}{2^{m}}$, isto é, para todo ,
$$x_{n_{m}}^{\ast}(x_{m})\geq (4^{-1}-2^{-m})\|x_{n_{m}}^{\ast}\|\geq M(4^{-1}-2^{-m})$$
for every $m\in\mathbb{N}$.%, então para todo $m\in\mathbb{N}$,  $x_{n_{m}}^{\ast}(x_{m})\geq M(4^{-1}-2^{-m})$.
\end{proof}

From now on we work with operators between Banach lattices.

%If $T$ is a positive operator, we get the following.

\begin{theorem}\label{equivalentes1} The following are equivalent
for a positive operator $T \colon E \to F$:\\
{\rm (a)} $T^{\ast} \colon F^{\ast}\rightarrow E^{\ast}$ is disjoint $p$-convergent.\\
%%{\rm (b)} $q_{T^{\ast}}(y_{n}^{\ast})\rightarrow 0$ for all disjoint weakly $p$-summable sequence $(y_{n}^{\ast})_n$ in $F^{\ast}$.\\
%%{\rm (c)} $q_{T^{\ast}}(y_{n}^{\ast})\rightarrow 0$ for all positive disjoint weakly $p$-summable sequence $(y_{n}^{\ast})_n$ in $F^{\ast}$.\\
{\rm (b)} $y_{n}^{\ast}(T(x_{n}))\longrightarrow 0$ for every disjoint weakly $p$-summable sequence $(y_{n}^{\ast})_n$ in $F^{\ast}$  and every bounded sequence  $(x_{n})_n$ in $E$.\\
{\rm (c)} $y_{n}^{\ast}(T(x_{n}))\longrightarrow 0$ for every disjoint weakly $p$-summable sequence $(y_{n}^{\ast})_n$ in $F^{\ast}$ and every bounded disjoint sequence  $(x_{n})_n$ in $E$.\\
{\rm (d)} $y_{n}^{\ast}(T(x_{n}))\longrightarrow 0$ for every positive disjoint weakly $p$-summable sequence $(y_{n}^{\ast})_n$ in $F^{\ast}$  and every bounded disjoint sequence  $(x_{n})_n$ in $E$.\\
{\rm (e)} $y_{n}^{\ast}(T(x_{n}))\longrightarrow 0$ for every positive disjoint weakly $p$-summable sequence $(y_{n}^{\ast})_n$ in $F^{\ast}$  and every bounded positive disjoint sequence  $(x_{n})_n$ in $E$.\\
{\rm (f)} Regardless of the Banach lattice $G$ and the Riesz homomorphism $S \colon G\rightarrow E$,  $S^{\ast}T^{\ast} \colon F^{\ast}\rightarrow G^{\ast}$ is a disjoint $p$-convergent operator.\\
{\rm (g)} Regardless of the Riesz homomorphism $S \colon \ell_{1}\rightarrow E$, $S^{\ast}T^{\ast} \colon F^{\ast}\rightarrow \ell_{1}^{\ast}$ is a disjoint $p$-convergent operator.\\
{\rm (h)} Regardless of the Banach lattice $G$ and the positive operator $S \colon G\rightarrow E$,  $S^{\ast}T^{\ast} \colon F^{\ast}\rightarrow G^{\ast}$ is a disjoint $p$-convergent operator.\\
{\rm (i)} Regardless of the positive operator $S \colon \ell_{1}\rightarrow E$, $S^{\ast}T^{\ast} \colon F^{\ast}\rightarrow \ell_{1}^{\ast}$ is a disjoint $p$-convergent operator.
\end{theorem}

\begin{proof}
The equivalence (a)$\Leftrightarrow$(b) was proved in Lemma \ref{equivalentes} and the implications (b)$\Rightarrow$ (c)$\Rightarrow$(d)$\Rightarrow$(e) and (f)$\Rightarrow$(g) are straightforward.

(e)$\Rightarrow$(a) If $T^\ast$ is not disjoint $p$-convergent, as we did in the proof of Lemma \ref{equivalentes} there exists a positive disjoint weakly $p$-summable sequence $(y_n^\ast)_n$ in $F^\ast$ such that  $\norma{T^\ast(y_n^\ast)} \geq \alpha$ for every $n \in \N$ and some $\alpha > 0$. Since $T^\ast$ is a positive weak*-weak* continuous operator, $(T^\ast y_n^\ast)_n$ is a positive weak* null sequence. By Lemma \ref{lemaantigo} there exist an increasing sequence $(n_k)_k \subseteq \N$ and a bounded positive disjoint sequence $(x_k)_k \subseteq E$ such that
 $T^{\ast}(y_{n_{k}}^{\ast})(x_{k})\geq \alpha(4^{-1}-2^{-k})$ for every $k\in\mathbb{N}$. From
$$ \alpha(4^{-1}-2^{-k})\leq T^{\ast}(y_{n_{k}}^{\ast})(x_{k})=y_{n_{k}}^{\ast}(T(x_{k})),$$
we obtain $\liminf\limits_{k \to \infty} y_{n_{k}}^{\ast}(T(x_{k})) \geq \alpha/4$, which contradicts (e).

(a)$\Rightarrow$(f) This implication follows from Lemma \ref{equivalentes}.

(g)$\Rightarrow$(e) Let $(y_n^\ast)_n$ be a positive disjoint weakly $p$-summable sequence in $F^\ast$ and let $(x_n)_n$ be a bounded positive disjoint sequence in $E$. It is clear that $S \colon \ell_1 \to E$ given by $S(a) = \sum\limits_{n=1}^\infty a_n x_n$ for every $a = (a_n)_n \in \ell_1$ is a well defined bounded linear operator. Using  the continuity of the lattice operations and that  $(x_n)_n$ is a positive disjoint sequence, we get
      $$|S(a)|=\Big|\lim_{m\rightarrow\infty}\displaystyle\sum_{n=1}^{m} a_n x_n \Big|=\lim_{m\rightarrow\infty}\Big|\displaystyle\sum_{n=1}^{m} a_n x_n\Big|=\lim_{m\rightarrow\infty}\displaystyle\sum_{n=1}^{m} |a_{n}| x_n =S(|a|)$$
      for every $a = (a_n)_n \in \ell_1$, proving that $S$ is a Riesz homomorphism. So $S^\ast T^\ast$ is disjoint $p$-convergent by assumption, which implies that
      $ |y_n^\ast(Tx_n)| \leq \norma{S^\ast T^\ast (y_n^\ast)} \longrightarrow 0. $

It is clear that the operator $S$ we have just worked with is positive, so the implication (i)$\Rightarrow$(e) follows analogously. The implication (a) $\Rightarrow$ (h) follows from Lemma \ref{equivalentes} and the implication (h) $\Rightarrow$ (i) is immediate.
%(i)$\Rightarrow$(e) Let $(y_n^\ast)_n$ be a positive disjoint weakly $p$-summable sequence in $F^\ast$ and let $(x_n)_n$ be a bounded positive disjoint sequence in $E$. Defining $S \colon \ell_1 \to E$ by $S(a) = \displaystyle \sum_{n=1}^\infty x_n a_n$ for every $a = (a_n)_n \in \ell_1$, we have that $S$ is a positive operator and so $S^\ast T^\ast$ is disjoint $p$-convergent. Therefore $ |y_n^\ast(Tx_n)| \leq \norma{S^\ast T^\ast (y_n^\ast)} \longrightarrow 0. $
\end{proof}

%%It was proved in \cite[Theorem 3.8]{ali} that if the adjoint of a disjoint $p$-convegernt operator $T: E \to F$ is also disjoint $p$-convergent then $E^\ast$ has order continuous norm or $F^\ast$ has the positive Schur property of order $p$. Since

Recall that a Banach lattice $E$ has the positive Schur property of order $p$ if every weakly $p$-summable sequence $(x_n)_n \subseteq E^+$ is norm null, or, equivalently, if every disjoint weakly $p$-summable sequence $(x_n)_n \subseteq E^+$ is norm null (see \cite[Proposition 3.3]{fourie}). We omit the (easy) proof of the following characterization.

\begin{lemma} \label{psp1}
    A Banach lattice $E$ fails the positive Schur property of order $p$ if and only if $E$ contains a positive normalized disjoint weakly $p$-summable sequence.
\end{lemma}

%\begin{proof}
%    If $E$ does not have the the positive Schur property of order $p$, there exists a disjoint weakly $p$-summable sequence $(y_n)_n \subseteq E^+$ which is not norm null. In particular, there exists a subsequence $(n_k)_k \subseteq \N$ such that $\norma{y_{n_k}} \geq \varepsilon$ for all $k \in \N$ and some $\varepsilon > 0$. Taking $x_k = y_{n_k} / \norma{y_{n_k}}$ for each $k \in \N$, we have that $(x_k)_k$ is a positive disjoint weakly $p$-summable sequence such that $\norma{x_k} = 1$ for all $k \in \N$. The converse is immediate.
%\end{proof}

In \cite[Theorem 3.8]{ali} it is proved that if the adjoint of every disjoint $p$-convergent operator $T \colon E \to F$ is disjoint $p$-convergent, then  $E^\ast$ has order continuous norm or $F^\ast$ has the positive Schur property of order $p$. By Lemma \ref{lema1}, the next result is an improvement of \cite[Theorem 3.8]{ali} for positive operators. %%\textcolor{red}{Acho que s\'o a partir daqui \'e necess\'ario trabalhar com operadores entre reticulados.}

\begin{theorem} \label{teo4}
     If the adjoint of every positive order weakly compact  operator $T \colon E \to F$ is disjoint $p$-convergent, then  $E^\ast$ has order continuous norm or $F^\ast$ has the positive Schur property of order $p$.
\end{theorem}

\begin{proof} Suppose that the norm of $E^\ast$ is not order continuous norm and that
     $F^{\ast}$ fails the positive Schur property of order $p$.
     On the one hand, as $E^\ast$ does not have order continuous norm, there exists a non-norm null positive disjoint order bounded sequence $(x_{n}^{\ast})_n$ in $E^{\ast}$ (see \cite[Theorem 2.4.2]{meyer}). Without loss of generality, we may assume that $\norma{x_n^\ast} = 1$ for every $n \in \N$,  thus we may find a sequence $(x_n)_n \subset S_{E}^+$ such that $x_n^\ast(x_n) \geq 1/2$ for every $n \in \N$.
     %%% The same argument used in the proof of \cite[Theorem 3.5]{botelhogarcia} shows that there exist a positive disjoint order bounded sequence $(x_{n}^{\ast})_n \subseteq E^{\ast}$ and a positive sequence $(x_n)_n\subseteq E$ such that $\norma{x_{n}^{\ast}} = \norma{x_n} = 1$ and $x_n^{\ast}(x_n) \geq 1/2$ for every $n \in \N$.
      Letting $x^* \in E^\ast$ be such that $0\leq x_{n}^{\ast}\leq x^*$ for every $n$, we have
     $$\sum_{n=1}^{m}|x_{n}^{\ast}(x)| \leq \sum_{n=1}^{m}x_{n}^{\ast}(|x|)=\Big(\bigvee_{n=1}^{m}x_{n}^{\ast}\Big)(|x|)\leq x^*(|x|)$$
     for every $m \in \N$ and every $x \in E$. Thus, $S \colon E\rightarrow \ell_{1}$ given by $S(x)=(x_{n}^{\ast}(x))_n$ is a well defined (bounded) linear operator.
     On the other hand, by the failure of the positive Schur property of order $p$ by $F^{\ast}$ and Lemma \ref{psp1}, there exists a positive disjoint weakly $p$-summable sequence $(y_n^{\ast})_n$ in $F^{\ast}$ such that $\norma{y_n^{\ast}} =1$ for all $n \in \N$. Moreover, for each $n \in \N$ there exists $y_n \in F$ such that $\|y_{n}\|=1$ and $y_{n}^{\ast}(y_{n})\geq \frac{1}{2}$. Defining $R \colon \ell_1 \rightarrow F$ by $R((a_{n})_n)=\sum\limits_{n=1}^{\infty}a_{n}y_{n}$ and $T := R \circ S$, we have that $T$ is a positive operator from $E$ to $F$ such that $T(x)=\sum\limits_{n=1}^{\infty}x_{n}^{\ast}(x)y_{n}$ for every $x \in E$. Since
    $$\sum_{n=1}^{m}\|x_{n}^{\ast}(x)y_{n}\|=\sum_{n=1}^{m}|x_{n}^{\ast}(x)|\cdot\|y_{n}\|\leq \sum_{n=1}^{m}x_{n}^{\ast}(|x|)=\Big(\bigvee_{n=1}^{m} x_{n}^{\ast}\Big)(|x|)\leq x^*(|x|)$$
     for all $x\in E$ and $m\in\N$, we get $\norma{T(x)} \leq x^*(|x|)$ for every $x\in E$.

    Let us prove that $T$ is order weakly compact. Given an order bounded disjoint sequence $(w_{n})_n$ in $E$, let $w,w_0 \in E$ be such that $w\leq w_n\leq w_0$ for every $n \in \N$. An elementary calculation gives $w\leq |w_{n}|\leq w_{0}+2|w|$, so the sequence $(|w_{n}|)_n$ is disjoint and order bounded in $E^{+}$. It follows that, for all $z^{\ast}\in E^{\ast}$ and $k\in\mathbb{N}$,
    $$\sum_{n=1}^{k}|z^{\ast}(|w_{n}|)|\leq \sum_{n=1}^{k}|z^{\ast}|(|w_{n}|)=|z^{\ast}|\Big(\sum_{n=1}^{k}(|w_{n}|)\Big)= |z^{\ast}|\Big(\bigvee_{n=1}^{k}(|w_{n}|)\Big)\leq |z^{\ast}|(w_{0}+2|w|),$$
    proving that $(z^{\ast}(|w_{n}|))_n \in \ell_{1}$. In particular, $(|w_{n}|)_n$ is weakly null in $E$, so $\|T(w_{n})\|\leq x^*(|w_{n}|)\longrightarrow 0$. By \cite[Theorem 5.57]{alip}, $T$ is order weakly compact.
    For any $k \in \N$ we have $x_k \geq 0$, $\|x_k\| = 1$ and
    $$ T^{\ast}(y_{k}^{\ast})(x_k)=\displaystyle\sum_{n=1}^{\infty}x_{n}^{\ast}(x_k)y_{k}^{\ast}(y_{n})\geq x_{k}^{\ast}(x_k)y_{k}^{\ast}(y_{k})\geq \frac14% \dfrac{x_{k}^{\ast}(x_k)}{2}\geq \frac12
    , $$
          so $\|T^{\ast}(y_{k}^{\ast})\|\geq \frac{1}{4}$ for every $k \in \N$, which implies that $T^\ast$ fails to be disjoint $p$-convergent.
\end{proof}

\begin{example}\rm Let us see that Theorem \ref{teo4} is a concrete improvement of
%Since, by Lemma \ref{lema1}, every disjoint $p$-convergent operator is order weakly compact, Theorem \ref{teo4} improves
\cite[Theorem 3.8]{ali} for positive operators. Indeed,  ${\rm id}_{c_0}$ is an order weakly compact operator which is not disjoint $p$-convergent such that $(c_0)^\ast = \ell_1$ has order continuous norm and the Schur property, hence the positive Schur property of order $p$.
\end{example}

Although we do not know whether the converse of Theorem \ref{teo4} holds we can prove a partial converse of \cite[Theorem 3.8]{ali}.

\begin{theorem}
    If $E$ has the disjoint property of order $p$ or $F^\ast$ has the positive Schur property of order $p$, then the adjoint of every positive disjoint $p$-convergent operator from $E$ to $F$ is disjoint $p$-convergent.
\end{theorem}

\begin{proof} Let $T \colon E \to F$ be a positive disjoint $p$-convergent operator.    Assume first that $E$  has the disjoint property of order $p$. It follows from Theorem \ref{teo3} that $T$ is an almost Dunford-Pettis operator. Moreover, since $E^\ast$ has order continuous norm (Example \ref{djp2}(3)), we obtain from \cite[Theorem 5.1]{aqz} that $T^\ast$ is almost Dunford-Pettis, hence disjoint $p$-convergent.

       Suppose now that $F^\ast$ has the positive Schur property of order $p$. If $(y_n^\ast)_n$ is a positive disjoint weakly $p$-summable sequence in $F^\ast$, it follows from the assumption that $\norma{y_n^\ast} \longrightarrow 0$, hence $\norma{T^\ast y_n^\ast} \longrightarrow 0$.
\end{proof}

%%\textcolor{blue}{A partir daqui tudo tem que ser lido com muito cuidado, por isso n\~ao colocarei mais azuis}

It is clear that the condition of $E^*$ having order continuous norm or $F^*$ having the positive Schur property of order $p$ plays a key role in the subject. Pursuing a characterization of this condition, in \cite{ali} an operator $T \colon E \to F$ between Banach lattices is said to be almost weak $p$-convergent if $y_n^\ast(Tx_n) \longrightarrow 0$ whenever $(x_n)_n$ is weakly null in $ E$ and $(y_n^\ast)_n$ is disjoint weakly $p$-summable in $ F^\ast$ (see \cite[Theorem 5.10]{ali}). In \cite[Theorem 5.11]{ali} it is proved a characterization when $E^*$ has order continuous norm or $F^\ast$ has the positive Schur property of order $p$ under the condition that every $p$-convergent operator from $E$ to $F$ is almost weak $p$-convergent. The final purpose of this paper is to obtain characterizations not depending on this condition.
%A natural question, then, is what is equivalent to the the thesis obtained in Theorem \ref{teo4}. One may look at \cite[Theorem 5.11]{ali} where it is proved that the adjoint of a positive operator $T \colon E \to F$ which is both $p$-convergent and almost weak $p$-convergent is disjoint $p$-convergent if and only if $E^\ast$ has order continuous norm or $F^\ast$ has the positive Schur property of order $p$. In Theorem \ref{teo5} below we will give another equivalent condition to this fact. Let us first recall the definition of an almost weak $p$-convergent operator.
First it is worth noting that the implication (a)$\Rightarrow$(d) of Lemma \ref{equivalentes} gives the following:

 \begin{lemma}\label{nlemn} If the adjoint of an operator between Banach lattices is disjoint $p$-convergent, then the operator  is almost weak $p$-convergent.
 \end{lemma}

Let us see that the converse of the implication above does not hold in general even for positive operators.

\begin{example} \label{exemplo}\rm
    %Letting $T = I_{\ell_1}$ be the identity operator on $\ell_1$, we have that $T$ is an almost weak $p$-convergent operator whose adjoint is not disjoint $p$-convergent.  Indeed, by one hand,
     On the one hand, for every weakly null sequence $(x_n)_n$ in $\ell_1$ and every disjoint weakly $p$-summable sequence $(y_n^\ast)_n$ in $(\ell_1)^\ast$, we have $ |y_n^\ast(x_n)| \leq \|{y_n^\ast}\|\cdot \norma{x_n} \longrightarrow 0 $
    because $(y_n^\ast)_n$ is bounded and $\ell_1$ has the Schur property. This shows that the identity operator ${\rm id}_{\ell_1}$ on $\ell_1$ is almost weak $p$-convergent.
        On the other hand, since the natural inclusion $ \ell_{p^{\ast}} \hookrightarrow \ell_\infty$ is not compact, $\ell_\infty$ fails the Schur property of order $p$ \cite[p.\,879]{fourie}, so $\ell_\infty$ fails the positive Schur property of order $p$ by \cite[Lemma 2.6 and Proposition 3.11]{fourie}. Therefore, ${\rm id}_{\ell_\infty} = ({\rm id}_{\ell_1})^*$ is not disjoint $p$-convergent.
\end{example}

%  $$T^* \mbox{ is disjoint $p$-convergent} \Longrightarrow T \mbox{ is almost weak $p$-convergent.}$$
 Our characterization (Theorem \ref{teo5}) shows that $E^*$ has order continuous norm or $F^\ast$ has the positive Schur property of order $p$ if and only if the converse of Lemma \ref{nlemn} holds for positive operators. Among other things, the proof depends on the the next lemma.

\begin{lemma} \label{lem2}
    A positive operator $T \colon E \to F$ is almost weak $p$-convergent if and only if $y_n^\ast(Tx_n) \longrightarrow 0$ whenever $(x_n)_n$ is disjoint weakly null in $E$ and $(y_n^\ast)_n$ is disjoint weakly $p$-summable sequence in $F^*$.
\end{lemma}

\begin{proof} Let us prove the nontrivial implication. Assuming that $T$ is not an almost weak $p$-convergent operator, there exist a weakly null sequence $(x_n)_n \subseteq E$, a disjoint weakly $p$-summable sequence $(y_n^\ast)_n \subseteq F^\ast$ and $\varepsilon > 0$  such that $|y_n^\ast(Tx_n)| \geq \varepsilon$ for all $n \in \N$. Since $(y_n^\ast)_n$ is disjoint weakly $p$-summable, $(|y_n^\ast|)_n$ is weakly $p$-summable sequence by \cite[Proposition 2.2]{fourie}, and, in particular $|y_n^\ast| \cvfe 0$ in $F^\ast$. Let $n_1 = 1$.
    As $|y_n^\ast|(T(4|x_{n_1}|))\longrightarrow 0$, there exists $n_2 > n_1$ such that $|y_{n_2}^\ast| (T(4 |x_{n_1}|)) < 1/2$. Again, since $|y_n^\ast|(T(4^2 \sum\limits_{j=1}^2 |x_{n_j}|)) \longrightarrow 0$, there exists $n_3 > n_2$ such that $|y_{n_3}^\ast| (T(4^2 \sum\limits_{j=1}^2 |x_{n_j}|)) < 1/2^2$. By induction, we construct a strictly increasing sequence $(n_k)_k \subseteq \N$ such that
    $$ |y_{n_{k+1}}^\ast| (T(4^k {\textstyle\sum\limits_{j=1}^k} |x_{n_j}|)) < \frac{1}{2^k} \text{~for every } k \in \N. $$
    Letting $x = \sum\limits_{k=1}^\infty 2^{-k} |x_{n_k}|$ and
    $$ u_k = \left ( |x_{n_{k+1}}| -4^k \textstyle\sum\limits_{j=1}^k |x_{n_j}| -2^{-k} x \right )^+, $$
    we obtain from \cite[Lemma 4.35]{alip} that $(u_k)_k$ is a positive disjoint sequence such that $u_k \leq |x_{n_{k+1}}|$ for every $k \in \N$. Since $A = \conj{x_n}{n \in \N}$ is relatively weakly compact and $(u_k)_k$ is a disjoint sequence contained in $\sol{A}$, we have $u_k \stackrel{\omega}{\longrightarrow} 0$ by \cite[Theorem 4.34]{alip}). By assumption, $ |y_{n_{k+1}}^\ast|(T(u_{k})) \longrightarrow 0$. For every $k \in \mathbb{N}$,
    \begin{align*}
        |y_{n_{k+1}}^\ast| (T(u_k))
        & \geq |y_{n_{k+1}}^\ast|(T(|x_{n_{k+1}}|)) - |y_{n_{k+1}}^\ast|(T(4^k {\textstyle\sum\limits_{j=1}^k} |x_{n_j}|)) - 2^{-k} |y_{n_{k+1}}^\ast|(T(x)) \\
        &\geq \varepsilon- \frac{1}{2^k} - 2^{-k} |y_{n_{k+1}}^\ast|(T(x)).
    \end{align*}
Letting $k \to \infty$ and using that $|y_{n_{k+1}}^\ast|(T(x)) \longrightarrow 0$ we get
$0 = \lim\limits_k |y_{n_{k+1}}^\ast| (T(u_k)) \geq \varepsilon, $ a contradiction which completes the proof.
%$ \displaystyle\liminf_{k \to \infty} |y_{n_{k+1}}^\ast| (T(u_k)) \geq \varepsilon, $ which contradicts $ |y_{n_{k+1}}^\ast|(T(u_{k})) \longrightarrow 0$.
\end{proof}

%Let $T \colon E \to F$ be a positive operator such that $T^\ast \colon F^\ast \to E^\ast$ is disjoint $p$-convergent. If $(x_n)_n \subseteq E$ is a weakly null sequence and if $(y_n^\ast)_n \subseteq F^\ast$ is a disjoint weakly $p$-summable sequence, we obtain that $\norma{T^\ast(y_n^\ast)} \longrightarrow 0$, and then
%$$ |y_n^\ast(T(x_n))| = |(T^\ast(y_n^\ast))(x_n)| \leq \norma{x_n}  \norma{T^\ast(y_n)} \longrightarrow 0$$
%because $\sup_{n \in \N} \norma{x_n} < \infty$. Therefore $T$ is an almost weak $p$-convergent operator.
%Let us see that the there exist almost weak $p$-convergent operators whose adjoint is not disjoint $p$-convergent:
%
%
%
%We now prove another equivalent condition to the thesis of Theorem \ref{teo4}.

 %%%In our next result we give necessary and sufficient conditions to establish when the adjoint of an almost weak $p$-convergent operator is disjoint $p$-convegrent.

%\textcolor{red}{Atenção. Se for verdade que todo operador $p$-convergente é almost weak $p$-convergente, então o teorema abaixo é menos geral que \cite[Theorem 5.11]{ali}.}

\begin{theorem} \label{teo5} The following are equivalent for two Banach lattices $E$ and $F$.\\
{\rm (a)} $E^\ast$ has order continuous norm or $F^\ast$ has the positive Schur property of order $p$.\\
{\rm (b)} The adjoint of every positive almost weak $p$-convergent operator from $E$ to $F$ is disjoint $p$-convergent.\\
{\rm (c)} A positive operator $T \colon E \to F$ is almost weak $p$-convergent if and only if $T^*$ is disjoint $p$-convergent.
    %%Let $E, F$ Banach lattices. The adjoint of every almost weak $p$-convergent operator $T: E \to F$ is disjoint $p$-convergent if and only if one of the following holds: \\
    %{\rm (a)} $E^{\ast}$ has order continuous norm.\\
    %{\rm (b)} $F^{\ast}$ has the positive Schur property of order $p$.
\end{theorem}
\begin{proof} Lemma \ref{nlemn}, which was derived from Lemma \ref{equivalentes}, gives the equivalence (b)$\Leftrightarrow$(c).

\medskip

\noindent (b)$\Rightarrow$(a) Assume, for sake of contradiction, that $E^\ast$ does not have order continuous norm and that $F^\ast$ fails to have the positive Schur property of order $p$. By the same argument used in the proof of Theorem \ref{teo4} there exist a positive disjoint order bounded sequence $(x_n^\ast)_n \subseteq E^\ast$, $\varphi \in E^\ast$ and a positive sequence $(x_n)_n \subseteq E$ such that $0 \leq x_n^\ast \leq \varphi$, $\norma{x_n^\ast} = \norma{x_n} = 1$ and $x_n^\ast(x_n) \geq 1/2$ for all $n \in \N$. Furthermore, $S(x) = (x_n^\ast(x))_n$ defines a positive linear operator $S \colon E \to \ell_1$ such that $\norma{S(x)} \leq \varphi(|x|)$ for every $x \in E$. As $F^\ast$ does not have the positive Schur property of order $p$, by Lemma \ref{psp1} there exists a positive disjoint weakly $p$-summable sequence $(y_n^\ast)_n$ in $F^\ast$ such that $\norma{y_n^\ast} = 1$ for every $n \in \N$. Moreover, for each $n \in \N$, there exists $y_n \in F$ such that $\norma{y_n} = 1$  and $y_n^\ast(y_n)  \geq 1/2$. Defining $T \colon E \to F$ by $T(x) = \sum\limits_{n=1}^\infty x_n^\ast(x) y_n$, we obtain that $T$ is a positive operator such that $\norma{T(x)} \leq \varphi(|x|)$
for every $x \in E$. We claim that $T$ is almost weak $p$-convergent.
To check this, let $(z_n)_n$ be a disjoint weakly null sequence in $E$ and  $(f_n)_n$ be a disjoint weakly $p$-summable sequence in $F^\ast$. We have
$$ |f_n(Tz_n)| \leq \norma{f_n} \cdot \norma{T (z_n)} \leq \sup_{n \in \N} \norma{f_n}\cdot \varphi(|z_n|) \longrightarrow 0 $$
 because $\sup\limits_{n \in \N}\norma{f_n} < \infty$  and $(|z_n|)_n$ is weakly null by \cite[Proposition 1.3]{wnukdual}. Lemma \ref{lem2} gives that $T$ is a almost weak $p$-convergent operator. To complete the reasoning, note that, since $(y_k^\ast)_k$  is a disjoint weakly $p$-summable sequence in $F$ such that
$$ y_k^\ast(T(x_k)) = \displaystyle\sum_{n=1}^{\infty} x_n^\ast(x_k)y_k^\ast(y_n) \geq x_k^\ast(x_k)y_k^\ast(y_k) \geq 1/4 $$
 for every $k \in \N$, we  have by Theorem \ref{equivalentes1} that $T^\ast$ is not disjoint $p$-convergent, which is a contradiction.

\medskip

\noindent (a)$\Rightarrow$(b) Let $T \colon E \to F$ be an almost weak $p$-convergent operator, $(y_n^\ast)_n$ be a disjoint weakly $p$-summable sequence in $F^\ast$ and $(x_n)_n$ be a bounded disjoint sequence in $E$.
If $E^\ast$ has order continuous norm, then the sequence $x_n \stackrel{\omega}{\longrightarrow} 0$ (see \cite[Theorem 2.4.14]{meyer}), and  $y_n^\ast(Tx_n) \longrightarrow 0$ by Lemma \ref{lem2}. So, Theorem \ref{equivalentes1} yields that $T^\ast$ is disjoint $p$-convergent. If $F^\ast$ has the positive Schur property of order $p$, then $\norma{y_n^\ast} \longrightarrow 0$, which implies that $$|y_n^\ast(T(x_n))| \leq \norma{y_n^\ast}\cdot \norma{T} \cdot \sup\limits_{n \in \N} \norma{x_n} \longrightarrow 0$$
because $\sup\limits_{n \in \N} \norma{x_n} < \infty$. Calling on Theorem \ref{equivalentes1} once again we conclude that $T^\ast$ is disjoint $p$-convergent.
\end{proof}

It is important to make clear the relationships between \cite[Theorem 5.11]{ali} and Theorem \ref{teo5}. On the one hand, as mentioned before, \cite[Theorem 5.11]{ali} applies only for Banach lattices $E$ and $F$ for which every $p$-convergent operator from $E$ to $F$ is almost weak $p$-convergent; and now we add that, in this case, \cite[Theorem 5.11]{ali} is better than Theorem \ref{teo5}. On the other hand, Theorem \ref{teo5} applies for all Banach lattices $E$ and $F$. So, Theorem \ref{teo5} is useful only if there exist Banach lattices $E$ and $F$ such that $E^\ast$ has order continuous norm or $F^\ast$ has the positive Schur property of order $p$ and positive disjoint $p$-convergent operators from $E$ to $F$ that fail to be almost weak $p$-convergent. Next we show that such lattices and operators do exist, thus establishing the usefulness of Theorem \ref{teo5}.

\begin{example}\rm Taking $p = \frac32$, since $p^* = 3 > \frac32$, ${\rm id}_{\ell_{\frac32}}$ is $\frac32$-convergent by \cite[Example 3.9]{fourie}. Using that the sequence $(e_n)_n$ of canonical unit vectors is weakly null in $\ell_{\frac32}$, the sequence
$(e_n^{\ast})_n$ of the corresponding coordinate functionals is disjoint weakly $\frac32$-summable in $(\ell_{\frac32})^* = \ell_3$ and $e_n^*(e_n) = 1$ for every $n$, it follows that ${\rm id}_{\ell_{\frac32}}$ fails to be almost weak $\frac32$-convergent. Of course $(\ell_{\frac32})^* = \ell_3$ has order continuous norm.
\end{example}

We finish the paper with a discussion about \cite[Corollary 5.12]{ali}. It is not difficult to see that, in the way it is stated, this result is false (again the counterexample is ${\rm id}_{\ell_{\frac32}}$). We believe the correct statement concerns positive almost {\it weak} $p$-convergent operators $T \colon E \to F$ and that, in this fashion, the corollary somehow follows from \cite[Theorem 5.11]{ali}. Anyway, this form of the corollary can be obtained directly from Theorem \ref{teo5}. Indeed,  as Banach lattices with the positive Schur property of order $p$ have order continuous norms, applying Theorem \ref{teo5} twice we get the following:%the following is an immediate application of Theorem \ref{teo5}.

\begin{corollary} The norm of the dual $E^*$ of a Banach lattice $E$ is order continuous if and only if the adjoint of every positive almost weak $p$-convergent operator $T \colon E \to E$ is disjoint $p$-convergent.\end{corollary}

%\begin{proof} \textcolor{red}{DEPOIS TIRAMOS}
%
%    Assume first that $E^\ast$ has order continuous norm and let $T \colon E \to E$ be an almost weak $p$-convergent operator. By Theorem \ref{teo5}, we obtain that $T$ is disjoint $p$-convergent.
%
%
%    Assume now that the adjoint of every almost weak $p$-convergent operator $T \colon E \to E$ is disjoint $p$-convergent. By Theorem \ref{teo5}, we get that or $E^\ast$ has order continuous norm or $E^\ast$ has the positive Schur property of order $p$. In both cases, $E^\ast$ has order continuous norm.
%\end{proof}

%citar o Ardakani, Ann Funct Analy como contribuicao recente para os operadores disjoint $p$-convergentes. ????????

\bigskip

\noindent G. Botelho and V. C. C. Miranda\\
Faculdade de Matem\'atica\\
Universidade Federal de Uberl\^andia\\
38.400-902 -- Uberl\^andia -- Brazil\\
e-mails: botelho@ufu.br, colferaiv@gmail.com

\medskip

\noindent L. A. Garcia\\
Instituto de Ci\^encias Exatas\\
Universidade Federal de Minas Gerais\\
31.270-901 – Belo Horizonte – Brazil\\
e-mail: garcia$\_$1s@hotmail.com

\end{document}